\documentclass[11pt,reqno]{amsart}
\usepackage{amssymb,latexsym}
\usepackage{graphicx}
\usepackage{url}		

\calclayout

\makeatletter
\newcommand{\monthyear}[1]{%
  \def\@monthyear{\uppercase{#1}}}
\newcommand{\volnumber}[1]{%
  \def\@volnumber{\uppercase{#1}}}
\AtBeginDocument{%
\def\ps@plain{\ps@empty
  \def\@oddfoot{\@monthyear \hfil \thepage}%
  \def\@evenfoot{\thepage \hfil \@volnumber}}
\def\ps@firstpage{\ps@plain}
\def\ps@headings{\ps@empty
  \def\@evenhead{%
    \setTrue{runhead}%
    \def\thanks{\protect\thanks@warning}%
    \uppercase{The Fibonacci Quarterly}\hfil}%
  \def\@oddhead{%
    \setTrue{runhead}%
    \def\thanks{\protect\thanks@warning}%
    \hfill\uppercase{On $X$-coordinates of Pell equations which are repdigits}}%
  \let\@mkboth\markboth
  \def\@evenfoot{%
    \thepage \hfil \@volnumber}%
  \def\@oddfoot{%
    \@monthyear \hfil \thepage}%
  }%
  \def\({\left(}
\def\){\right)}
\footskip=25pt
\pagestyle{headings}%
}
\makeatother

\theoremstyle{plain}
\numberwithin{equation}{section}
\newtheorem{thm}{Theorem}[section]
\newtheorem{theorem}[thm]{Theorem}
\newtheorem{lemma}[thm]{Lemma}

\newtheorem{problem}[thm]{Problem}

\begin{document}
\monthyear{Month Year}
\volnumber{Volume, Number}
\setcounter{page}{1}

\title{On $X$-coordinates of Pell equations which are repdigits}
\author{Bernadette Faye}
\address{
Ecole Doctorale de Math\'ematiques et d'Informatique \newline
Universit\'e Cheikh Anta Diop de Dakar \newline
BP 5005, Dakar Fann, Senegal \newline
School of Mathematics, University of the Witwatersrand \newline
Private Bag X3, Wits 2050, South Africa}
\email{bernadette@aims-senegal.org} 
\thanks{We thank the anonymous referee for a careful reading of the manuscript. 
B. F. thanks OWSD and Sida (Swedish International Development Cooperation Agency) for a scholarship during her Ph.D. studies at Wits. She also thanks Mouhamed M. Fall of AIMS-Senegal for useful conversations. Parts of this research work was done when F. L. visited the Max Planck Institute for Mathematics in Bonn, Germany in 2017.  He wants to thank the institution for the invitation. He is also supported by grants CPRR160325161141 and an A-rated researcher award both from the NRF of South Africa and by grant no. 17-02804S of the Czech Granting Agency.}
\author{Florian Luca}
\address{
School of Mathematics, University of the Witwatersrand \newline
Private Bag X3, Wits 2050, South Africa and\newline
Department of Mathematics, Faculty of Sciences \newline 
University of Ostrava 30\newline 
dubna 22, 701 03 Ostrava 1, Czech Republic}
\email{Florian.Luca@wits.ac.za}

\begin{abstract}
Let $b\ge 2$ be a given integer. In this paper, we show that there only finitely many positive integers $d$ which are not squares, such that the Pell equation $X^2-dY^2=1$ has two positive integer solutions $(X,Y)$ with the property  that their $X$-coordinates are base $b$-repdigits. Recall that a base $b$-repdigit is a positive integer all whose digits have the same value when written in base $b$. We also give an upper bound on the largest such $d$ in terms of $b$.
\end{abstract}

\maketitle

\section{Introduction}
Let $d>1$ be a positive integer which is not a perfect square. It is well known that the Pell equation

\begin{equation}
\label{eq:1}
X^2-dY^2=1
\end{equation}
has infinitely many positive integer solutions $(X,Y)$. Furthermore, putting $(X_1,Y_1)$ for the smallest solution, all solutions are of the form $(X_n,Y_n)$ 
for some positive integer $n$, where 
$$ 
X_n+\sqrt{d}Y_n = (X_1 +\sqrt{d}Y_1)^n.
$$
There are many papers in the literature which treat Diophantine equations involving members of the sequences $\{X_n\}_{n\geq 1}$ and/or $\{Y_n\}_{n\geq 1}$,
such as when are these numbers squares, or perfect powers of fixed or variable exponents of some others positive integers, or Fibonacci numbers, etc. (see, for example, \cite{BMS}, \cite{Cohn1}, \cite{Cohn2}, \cite{FL}, \cite{FT}). 
Let $b\geq 2$ be an integer. A natural number $N$ is called a {\it base $b$-repdigit} if all of its base $b$-digits are equal. Setting $a\in \{1,2,\ldots,b-1\}$ for the common value 
of the repeating digit and $m$ for the number of base $b$-digits of $N$, we have
\begin{equation}
\label{eq:2}
N=a\left(\frac{b^m-1}{b-1}\right). 
\end{equation}
When $a=1$, such numbers are called {\it base $b$-repunits}. When $b=10$, we omit mentioning the base and simply say that $N$ is a {\it repdigit}. 
In \cite{LAA}, A. Dossavi-Yovo, F. Luca and A. Togb\'e proved  that when $d$ is fixed there is at most one $n$ such that $X_n$ is a repdigit except when $d=2$ (for which both $X_1=3$ and $X_3=99$ 
are repdigits) or $d=3$ (for which both $X_1=2$ and $X_2=7$ are repdigits). 
In this paper, we prove that the analogous result holds if we replace ``repdigits" by ``base $b$-repdigits", namely that 
there is at most one $n$ such that $X_n$ is a base $b$-repdigit except for finitely many $d$, and give an explicit bound 
depending on $b$ on the largest possible exceptional $d$. 

Of course, for every integer $X\geq 2$, there is a unique square-free integer $d\geq 2$  such that 
\begin{equation}
\label{eq:31}
X^2 -1=dY^2
\end{equation}
for some positive integer $Y$. In particular, if we start with an $N$ given as in \eqref{eq:2}, then $N$ is the $X$-coordinate of the solution to the Pell equation corresponding 
to the number $d$ obtained as in \eqref{eq:31}. Thus the problem becomes interesting only when we ask that the Pell equation corresponding to some $d$ (fixed or variable) 
has at least two positive integer solutions  whose $X$-coordinates are base $b$-repdigits.

Here, we apply the method from \cite{LAA}, together with an explicit estimate on the absolute value of the largest integer solution to an elliptic equation and prove the following result. 
\begin{theorem}
\label{thm:main}
Let $b\ge 2$ be fixed. Let $d\geq 2$ be squarefree and let $(X_n,Y_n):=(X_n(d),Y_n(d))$ be the $n$th positive integer solution of the Pell equation $X^2-dY^2=1$. If the Diophantine equation 
\begin{equation}
\label{eq:3}
X_n=a\left(\frac{b^m-1}{b-1}\right) \qquad {\text{with}}\qquad  a\in \{1,2,\ldots,b-1 \}
\end{equation}
has two positive integer solutions $(n,a,m)$, then 
\begin{equation}
\label{eq:conc}
d\leq \exp\left((10b)^{10^5}\right).
\end{equation}
\end{theorem}
 
The proof proceeds in two cases according to whether $n$ is even or odd. If $n$ is even, we reduce the problem to the study of integer points on some elliptic curves of a particular form. Here, we use an upper bound on the naive height 
of the  largest such point due to Baker. When $n$ is odd, we use lower bounds for linear forms in complex and $p$-adic logarithms. For a number field ${\mathbb K}$, a nonzero algebraic number $\eta\in {\mathbb K}$ and a prime ideal $\pi$ of ${\mathcal O}_{\mathbb K}$, we use $\nu_\pi(\eta)$ for the exact exponent of $\pi$ in the factorization in prime ideals of the principal fractional ideal $\eta {\mathcal O}_{\mathbb K}$ generated by $\eta$ in ${\mathbb K}$. When ${\mathbb K}={\mathbb Q}$ is the field of rational numbers and $\pi$ is some prime number $p$, then $\nu_p(\eta)$ coincides with the exponent of $p$ in the factorization of the rational number $\eta$.  

\section{Case when some  $n$ is even}

Assume that $n$ satisfies \eqref{eq:3}. Put $n=2n_1$. Then using known formulas for the solutions to Pell equations, \eqref{eq:3} implies that
\begin{equation}
\label{eq:4}
2X_{n_1}^2-1=X_{2n_1}=X_{n}=a\(\frac{b^m-1}{b-1}\).
\end{equation}
Assume first that $a=b-1$. Then \eqref{eq:4} gives
\begin{equation}
\label{eq:misodd}
2X_{n_1}^2 = b^m.
\end{equation}
We first deduce that $b$ is even. If $m=1$, then 
$$
d\le d Y_{n_1}^2=X_{n_1}^2-1<2X_{n_1}^2=b,
$$
so $d<b$, which is a much better inequality than the one we aim for in general.

From now on, we assume that  $m>1$. Then $X_{n_1}^2=b^m/2$. Since $m>1$, it follows that $X_{n_1}$ is even. This shows that $n_1$ is odd, for otherwise, if 
$n_1=2n_2$ is even, then $X_{n_1}=X_{2n_2}=2X_{n_2}^2-1$ is odd, a contradiction. Furthermore, the prime factors of $X_{n_1}$ are exactly all the prime factors of $b$. 
Let us show that $n$ is then unique. Indeed, assume that there exists $n'=2n_1'$ such that $(n',b-1,m')$ is a solution of \eqref{eq:3} and $n'\ne n$. Then also $X_{n_1'}^2=b^{m'}/2$, so $X_{n_1}$ and $X_{n_1'}$ have the same set of prime factors.  
Since $X_{n_1}=Y_{2n_1}/(2Y_{n_1})$ and $X_{n_1'}=Y_{2n_1'}/(2Y_{n_1'})$, the conclusion that $X_{n_1}$ and $X_{n_1'}$ have the same set of prime factors is false if $\max\{n_1,n_1'\}\ge 7$ by Carmichael's Theorem on Primitive Divisors for the sequence $\{Y_s\}_{s\ge 1}$ (namely that $Y_k$ has a prime factor not dividing any $Y_s$ for any $s<k$ provided that $k\ge 13$, see \cite{Ca}). Thus, $n_1,~n_1'\in \{1,3,5\}$ and one checks by hand that 
no two of $X_1,~X_3,~X_5$ can have the same set of prime factors. 

Thus, $n$ is unique and noting that $m$ is odd (for example, because the exponent of $2$ in the left--hand side of \eqref{eq:misodd} is odd), we get that $2^{(m-1)/2}$ divides $X_{n_1}$. It is known from the theory of Pell equations that $\nu_2(X_{n_1})=\nu_2(X_1)$. Hence, $X_1$ is a multiple of $2^{(m-1)/2}$. Further, since we are assuming that 
equation \eqref{eq:3} has two solutions $(n,a,m)$, it follows that there exists another solution either with $n$ even and $a\ne b-1$, or with $n$ odd. 

We now move on to analyze the case in which there exists a solution with $n$ even and $a\ne b-1$. 

\medskip

Put $m=3m_0+r$ with $r\in\{0,1,2\}$. Here, $m_0$ is a non-negative integer. Putting $x:=X_{n_1}$ and $y:=b^{m_0}$, equation \eqref{eq:4} becomes
\begin{equation}
\label{ap:7}
2x^2 -1=a\(\frac{b^ry^{3}-1}{b-1}\) .
\end{equation}
Equation \eqref{ap:7} leads to
\begin{equation}
\label{eq:ell}
X^2=Y^3+A_0,
\end{equation}
where 
$$
X:=4a(b-1)^2 b^r x,\qquad Y:=2a(b-1)b^r y,\qquad A_0:=8a^2(b-1)^3b^{2r}((b-1)-a).
$$
If $r=0$, then $A_0<2b^6\le 0.25b^{10}$ (here, we used the fact that $a(b-1-a)\le ((b-1)/2)^2<b^2/4$, a consequence of the AGM inequality). 
If $r\in \{1,2\}$, then since one of $b$ or $b-1$ is even we get that 
\begin{equation}
\label{eq:ell1}
X'^2=Y'^3+A_0',
\end{equation}
holds with integers $(X',Y',A_0')=(X/2^3,Y/2^2,A_0/2^6)$ and $A_0'=A_0/2^6<0.25 b^{10}$. Note that $A_0A_0'\ne 0$. 
Let us now recall the following result of Baker (see \cite{AB}).
  
\begin{theorem}
\label{Baker}
Let $A_0\neq 0$. Then  all integer solutions $(X,Y)$ of \eqref{eq:ell} satisfy 
 \begin{equation*}
 \label{eq:thbound}
 \max\{|X|,|Y|\}< \exp\{(10^{10}|A_0|)^{10^4}\}.
 \end{equation*}
\end{theorem} 
We will apply the above theorem to equation \eqref{eq:ell} for $r=0$ and to equation \eqref{eq:ell1} for $r\in \{1,2\}$. Note that since 
$|A_0|< 0.25\cdot b^{10}$ when $r=0$ and $|A_0'|<0.25b^{10}$ when $r\in \{1,2\}$, we get that 
$$
(10^{10} |A_0|)^{10^4}\le (0.25\cdot 10^{10} b^{10})^{10^4}<0.25 (10b)^{10^5},
$$
and a similar inequality holds for $A_0$ replaced by $A_0'$. Theorem \ref{Baker} applied to equations \eqref{eq:ell}, \eqref{eq:ell1} tells us that in both cases
$$
X/2^3<\exp(0.25 (10b)^{10^5}).
$$
Since $X\ge X_{n_1}>{\sqrt{d}}$, we get that 
$$
d<2^6\exp(0.5(10b)^{10^5})<\exp((10b)^{10^5}),
$$
which is what we wanted. So, let us conclude this section by summarizing what we have proved.

\begin{lemma}
\label{lem:2}
Assume that there is a solution $(n,a,m)$ to equation \eqref{eq:3} with $n$ even. Then one of the following holds:
\begin{itemize}
\item[(i)] $a<b-1$ and 
$$
d<\exp((10b)^{10^5}).
$$
\item[(ii)] $a=b-1$, $m=1$, and $d<b$.
\item[(iii)] $a=b-1$, $m>1$. Then $b$ is even, $m$ is odd and $n=2n_1$ is unique with this property, $X_{n_1}=b^{m}/2$ and $2^{(m-1)/2}$ divides $X_1$.
\end{itemize}
\end{lemma}

\section{ \small On the greatest divisor of repdigits}

From now on, we look at the case when equation \eqref{eq:3} has solutions $(n,a,m)$ with $n$ odd. Say,
$$
X_n=a\left(\frac{b^m-1}{b-1}\right).
$$
If some other solution $(n',a',m')$ to equation \eqref{eq:3} does not have $n'$ odd, then $n'$ must be even so Lemma \ref{lem:2} applies to it. If we are in one of the instances (i) or (ii) of Lemma \ref{lem:2}, then we are done. So, let us assume that we are in instance (iii) of Lemma \ref{lem:2}, so $n'$ is even, $m$ is odd and $2^{(m'-1)/2}$ divides $X_1$.
But since $n$ is odd, $X_1$ also divides $X_n$. Since $b$ is even, $(b^m-1)/(b-1)$ is odd, so $2^{(m'-1)/2}$ divides $a$. Hence, $2^{m'-1}\le 2a^2\le b^3$. Since $m'\le 2^{m'-1}$, we get that $m'\le b^3$. Writing $n'=2n_1'$, and using again (iii) of Lemma \ref{lem:2}, we get that
$$
d<X^2_1\le X^2_{n_1'}=b^{m'}/2<b^{m'}<b^{b^3}<\exp(b^4),
$$
which is good enough for us. 

From now on, we assume that both solutions of equation \eqref{eq:3} have an odd value for $n$. Let such indices be $n_1 \neq n_2$. Then
$$ 
X_{n_1}=a_1\left(\frac{b^{m_1}-1}{b-1}\right),\quad X_{n_2}=a_2\left(\frac{b^{m_2}-1}{b-1}\right), \quad {\text{\rm where}}\qquad a_1,a_2\in \{1,2,\ldots,b-1 \}. 
$$
Let $n_3:=\gcd(n_1,n_2).$ Since $n_1$ and $n_2$ are odd, from known properties of solutions to the Pell equation, we get that 
$$
X_{n_3}=\gcd(X_{n_1},X_{n_2}).
$$ 
We put $a_3:=\gcd(a_1,a_2),$ $a'_1:=a_1/a_3, a'_2:=a_2/a_3.$ We also put $m_3:=\gcd(m_1,m_2)$ and use the fact that 
$$
\gcd(b^{m_1} -1,b^{m_2}-1)=b^{m_3}-1.
$$ 
We then get,

\begin{eqnarray}
\label{eq:41}
X_{n_3} &=& \gcd(X_{n_1},X_{n_2})\nonumber\\
& = & \gcd\(a_1\frac{b^{m_1}-1}{b-1},a_2\frac{b^{m_2}-1}{b-1}\) \nonumber\\
&=& a_3\left(\frac{b^{m_3}-1}{b-1}\right)\gcd\(a'_1\frac{b^{m_1}-1}{b^{m_3}-1}, a'_2\frac{b^{m_2}-1}{b^{m_3}-1}\)\nonumber\\
& := & a_3c\left(\frac{b^{m_3}-1}{b-1}\right). 
\end{eqnarray}
The quantities inside the greatest common divisor denoted by $c$ have the properties that $a'_1,a'_2$ are coprime, $(b^{m_{1}}-1)/(b^{m_3}-1)$ and $(b^{m_{2}}-1)/(b^{m_3}-1)$ are also coprime. Thus,
$$
c=\gcd\left(a_1' , \frac{b^{m_2}-1}{b^{m_3}-1}\right)\gcd\left(\frac{b^{m_1}-1}{b^{m_3}-1}, a_2'\right),
$$
which implies that $c \le a_1' a_2'=(a_1a_2)/a_3^2$. Hence, $a_3c \le (a_1a_2)/a_3\le (b-1)^2$. Replacing $a_3 c$ by $a_3$, we retain the conclusion that 
$$
X_{n_3}=a_3\left(\frac{b^{m_3}-1}{b-1}\right),  \quad {\text{\rm where}}\quad a_3 \in \{1,2,3,\ldots,(b-1)^2\}.
$$
Since $n_1\neq n_2,$ we may assume that $n_1<n_2$, and then $n_3<n_2$ and $n_3$ is a proper divisor of $n_2$. Putting $n:=n_2/n_3, D:=X_{n_3}^2 - 1\ge d, m:=m_3, \ell:=m_2/m_3$ and relabeling $a_2$ and $a_3$ as $c$ and $a$, respectively, we can restate the problem now as follows:

\begin{problem}
\label{prob}
What can you say about $D$ such that 
\begin{eqnarray}
\label{eq:5}
X_{1} &=& a\left(\frac{b^m-1}{b-1}\right), \quad {\text{with}}\quad a  \in \{1,2,3,\ldots, (b-1)^2\}\nonumber \\
X_{n} &=& c\left(\frac{b^{m\ell}-1}{b-1}\right), \quad {\text{with}}\quad c \in \{1,2,3,\ldots,b-1\},
\end{eqnarray}
where $n>1$ is odd and $\ell,m$ are positive integers.
\end{problem}
From now on, we work with the system \eqref{eq:5}. By abuse of notation, we continue to use $d$ instead of $D$.

\section{Bounding $\ell$ in terms of $n$}

We may assume that $m\ge 100$ otherwise
$$
{\sqrt{d}}<X_1<b^{m+1}\le b^{101},
$$
so $d<b^{202}$, which is better than the inequality \eqref{eq:conc}.  We put
$$
\alpha:=X_1 +\sqrt{X_1^2-1}=X_1 + \sqrt{d}Y_1 .
$$
On one hand, from the first relation of \eqref{eq:5}, we have that
$$ 
X_1=\frac{1}{2}\left(\alpha + \alpha^{-1}\right)=a\left(\frac{b^m-1}{b-1}\right),
$$
so
\begin{equation}
\label{eq:6}
\alpha=X_1 + \sqrt{d}Y_1>X_1=\frac{1}{2}\left(\alpha + \alpha^{-1}\right)= a\left(\frac{b^m-1}{b-1}\right)>b^{m-1}.
\end{equation}
Hence, $ \alpha> b^{m-1}$ implying 
$$ 
m-1< \frac{\log \alpha }{\log b}.
$$
One the other hand, 
\begin{equation}
\label{eq:7}
\frac{X_1 + \sqrt{d}Y_1}{2}= \frac{\alpha}{2} <\frac{1}{2}(\alpha + \alpha^{-1}) = a\left(\frac{b^m-1}{b-1}\right)\le a(b^m-1)<ab^m<b^{m+2}.
\end{equation}
Hence, $ \alpha<2b^{m+2}\le b^{m+3}$ implying 
$$
\frac{\log\alpha}{\log b}<m+3. 
$$
Thus,
\begin{equation}
\label{eq:8}
m-1<\frac{\log\alpha}{\log b}<m+3.
\end{equation}
We now exploit the second relation of \eqref{eq:5}. On one hand, we get that
\begin{equation}
 \label{eq:9}
\alpha^n >X_n=c\left(\frac{b^{m\ell}-1}{b-1}\right)> b^{m\ell-1},
 \end{equation}
therefore
$$m\ell-1<n\left(\frac{\log \alpha}{\log b}\right)<n(m+3),
$$
where the last inequality follows from \eqref{eq:8}. Since $m$ is large ($m\ge 100$), we certainly get that
\begin{equation}
\label{eq:10}
\ell<2n +1.
\end{equation} 
On the other hand, 
$$
 \frac{\alpha^n}{2}<\frac{1}{2}\left(\alpha^n + \alpha^{-n}\right) =X_n <b^{m\ell} <(b\alpha)^\ell,
 $$
where the last inequality follows from \eqref{eq:6}. The  above inequalities lead to 
\begin{equation}
\label{eq:11}
\alpha^n<(2b\alpha)^\ell<(\alpha^{3/2})^\ell,
\end{equation}
where we used the fact that $\alpha^{1/2} >2b$,  or $\alpha>4b^2$, which follows from the fact that $\alpha>b^{m-1}$ together with the fact that $m\geq 100$. Inequality \eqref{eq:11} yields
$$ 
\ell>2n/3,
$$
(in particular, $\ell>1$ since $n\ge 3$), which together with \eqref{eq:10} gives 

\begin{equation}
\label{eq:12}
2n/3 < \ell< 2n+1.
\end{equation}
 
\section{Bounding $m$ in terms of $n$}

Here, we use the Chebyshev polynomial $P_{n}(X)\in \mathbb{Z}[X]$ for which $P_{n}(X_1)=X_n$. Recall that 
 $$ 
 P_n(X)=\frac{1}{2}\left((X + \sqrt{X^2-1})^n + (X-\sqrt{X^2-1})^n\right).$$ 
Using the second relation of \eqref{eq:5}, we have by Taylor's formula:
\begin{eqnarray}
\label{eq:13}
(c/(b-1)) b^{m\ell}-c/(b-1) & = & X_n\nonumber\\ 
& = & P_n(X_1)\nonumber\\
& = & P_n\left((a/(b-1))b^m-a/(b-1)\right)\\
& = & P_n(-a/(b-1))+P_n'(-a/(b-1))(a/(b-1))b^m\pmod {b^{2m}} \nonumber
\end{eqnarray}
Here, $1/(b-1) \pmod {b^m}$ is to be interpreted as the multiplicative inverse of $b-1$ modulo $b^m$, which exists since $b-1$ and $b$ are coprime.   

\medskip

{\bf Case 1.}  Suppose that $a=b-1$. 

\medskip

Then $X_1 +1=b^m$. Put $Y=X+1$ and denote 
$$
Q_n(Y):=P_n(Y-1)= \frac{1}{2}\left(\left(Y-1 + {\sqrt{(Y-1)^2-1}}\right)^n + \left(Y-1-{\sqrt{(Y-1)^2-1}}\right)^n\right).
$$
It was proved in \cite{LAA} that
$$Q_n(0)=-1\quad {\text{\rm and}}\quad \frac{dQ_n(Y)}{dY}\Big|_{Y=0}=Q'_n(0)=n^2.
$$
By Taylor's formula again, we get
$$ 
P_n(X)=Q_n(X+1)=n^2(X+1)-1 \quad\quad \pmod{ (X+1)^2}.
$$
Specializing at $X=X_1$ and $a=b-1$ and using the fact that $\ell>1$ (see \eqref{eq:12}), equation \eqref{eq:13} becomes
$$-c/(b-1)\equiv X_n \pmod{ b^{2m}}  \equiv -1 +n^2b^m \pmod{ b^{2m}} .
$$
If $c\neq b-1$, we get that $b^m\mid b-1-c$. Since $m\geq 100$ and $c<b^2$ and $c\ne b-1$, we get that $b^{100}\le |b-1-c|<b^2$, a contradiction.
If $c=b-1$, then
$$
-1\equiv -1+n^2b^m \pmod{b^{2m}}.
$$
The above congruence implies that 
\begin{equation}
\label{eq:n2}
b^m \mid n^2.
\end{equation}
 Since $n$ is odd, $b$ is odd also. Hence, $b\ge 3$. Thus, 
\begin{equation}
\label{eq:*}
b^m \le n^2 \quad {\text{\rm therefore}}\quad m\le 2\left(\frac{\log n}{\log b}\right)< 2\log n. 
\end{equation}

\medskip

{\bf Case 2.} {Suppose that $a\ne b-1$.}

\medskip

We put 
$$
\beta:=-a/(b-1) + \sqrt{(-a/(b-1))^2 -1}. 
$$
Equation \eqref{eq:13} gives
\begin{equation}
\label{eq:14}
b^m \mid P_n(-a/(b-1)) + c/(b-1),
\end{equation}
where the above divisibility is to be interpreted that $b^m$ divides the numerator of the rational number $P_n(-a/(b-1))+c/(b-1)$ written in reduced form.
We  observe that 
$$
P_n(-a/(b-1)) + c/(b-1)=(1/2) \beta^{-n}(\beta^n -\gamma)(\beta^n - \gamma^{-1}), 
$$
where $\gamma:= -c/(b-1) + \sqrt{(-c/(b-1))^2-1}$. Hence, 
\begin{equation}
\label{eq:16}
 b^m\mid (\beta^n-\gamma)(\beta^n-\gamma^{-1}). 
\end{equation}
It could be however that the right--hand side of \eqref{eq:16} is zero and then divisibility relation \eqref{eq:16} is not useful. In that case, we return to \eqref{eq:13}, use the fact that $P_n(-a/(b-1))+c/(b-1)=0$, to infer that 
\begin{equation}
\label{eq:20}
b^m\mid P_n'(-a/(b-1)).
\end{equation} 
Calculating we get
$$
P_n'(X)=\frac{n}{\sqrt{X^2-1}} \left((X+{\sqrt{X^2-1}})^n-(X-{\sqrt{X^2-1}})^n\right).
$$
Thus,
$$
b^m\mid \frac{n (b-1)}{{\sqrt{a^2-(b-1)^2}}} \left(\beta^n-\beta^{-n}\right),
$$
so 
\begin{equation}
\label{eq:21}
b^m\mid n(\beta^n-\beta^{-n}).
\end{equation} 
To continue, we need the following lemma.

\begin{lemma}
\label{lem:3}
The simultaneous system of equations $ \beta^{-n}=\beta^n=\gamma^i$ for some $i\in \{\pm 1\}$
has no solutions.
\end{lemma}

\begin{proof}
If $\beta^n=\beta^{-n}$, then $\beta^{2n}=1$. Hence, $\beta$ is a root of unity of order dividing $2n$. Since $\beta$ is rational or quadratic and $n$ is odd, 
we deduce that the order of $\beta$ is $1,2,3,6$. If the order of $\beta$ is $1,2$, we then get $\beta=\pm 1$. With $x:=-a/(b-1)$, we get
$$
x+{\sqrt{x^2-1}}=\pm 1,
$$
whose solutions are $x=\pm 1$. This leads to $a=\pm (b-1)$, which is false because $1\le a<b-1$. Hence, the order of $\beta$ is $3,6$. It follows that 
$$
x+{\sqrt{x^2-1}}=\pm (1/2\pm {\sqrt{3}} i/2).
$$
This gives $x=\pm 1/2$. Hence, $-a/(b-1)=\pm 1/2$, giving that $a=(b-1)/2$, $b$ is odd and $\beta=-(1/2\pm i{\sqrt{3}}/2)$.  Thus, $\beta^n=1$, showing that $\gamma=1$. With $y=-c/(b-1)$, we get $y+{\sqrt{y^2-1}}=1$, giving $y=1$.
Thus, $c=-(b-1)$, a contradiction.  
\end{proof}

We summarize what we did so far.

\begin{lemma}
\begin{itemize}
\item[(i)] If $a=b-1$, then 
$$
b^m\mid n^2.
$$
\item[(ii)] If $a\ne b-1$, then 
\begin{equation}
\label{eq:a1}
b^m\mid (\beta^n-\gamma)(\beta^n-\gamma^{-1}).
\end{equation}
Further, the expression appearing in the right-hand side of divisibility relation \eqref{eq:a1} either is nonzero, or it is zero in which case we additionally have
\begin{equation}
\label{eq:a2}
b^m\mid n(\beta^{n}-\beta^{-n}),
\end{equation}
and the expression appearing in the right--hand side of \eqref{eq:a2} is nonzero.
\end{itemize}
\end{lemma}
Thus, 
\begin{equation}
\label{eq:30}
b^m\mid n^2\Lambda,
\end{equation}
where either $\Lambda=1$ or
$$
\Lambda=\prod_{i=1}^s (\delta_i^{d_i}-\eta_i^{e_i}),
$$
for some $s\in \{1,2\}$, $(\delta_i,\eta_i)\in \{(\beta,\gamma),(\beta,\beta)\}$, $(d_i,e_i)\in \{(n,1),(n,-1),(n,n)\}$ for $1\le i\le s$ and furthermore $\Lambda\ne 0$.  
Let ${\mathbb K}={\mathbb Q}[\beta,\gamma]$ be of degree $D$. Note that $D\le 4$. Let $p$ be any prime factor of $b$ and let $\pi$ be some prime ideal in ${\mathbb K}$ dividing $p$. Then \eqref{eq:*} and \eqref{eq:21} tell us that 
$$
m\le 2\max\{\nu_{\pi}(\delta_i^{d_i}-\eta_i^{e_i}): 1\le i\le s\}+2D\nu_p(n).
$$
Note further that both $\beta$ and $\gamma$ are invertible modulo any prime dividing $b$. Indeed this follows, for example, because
$$
\beta=\frac{\lambda_1}{b-1}\quad {\text{\rm and}}\quad \beta^{-1}=\frac{\lambda_2}{b-1},
$$
where $\lambda_{1,2}=-a\pm {\sqrt{a^2-(b-1)^2}}$ are algebraic integers. Thus, $\lambda_1\lambda_2=b-1$, showing that if $\pi$ is any prime ideal such that one of 
$\nu_{\pi}(\lambda_1)$ or $\nu_{\pi}(\lambda_2)$ is nonzero, then $\pi\mid b-1$. In particular, $\pi\nmid b$.  A similar argument applies to $\gamma$.

Now, we use a linear form in $p$-adic logarithm due to K. Yu \cite{KY}, to get an upper bound for $m$ in terms of $n$. We recall the statement of Yu's result.

\begin{theorem}
\label{th1}
Let $\alpha_1,\ldots,\alpha_t$ be algebraic numbers in the field $\mathbb{K}$ of degree $D$, and 
$b_1,\ldots,b_t$ be nonzero integers. Put 
$$\Lambda=\alpha_1^{b_1}\ldots \alpha_t^{b_t}-1 $$
and
$$ B\geq \max\{|b_1|,\ldots,|b_t|\}. $$
Let $\pi$ be a prime ideal of ${\mathbb K}$ sitting above the rational prime $p$ of ramification $e_{\pi}$ and
$$ H_i\geq \max\{h(\alpha_i),\log p\} \quad\quad \hbox{for $i=1,\ldots,t$},$$
where $h(\eta)$ is the Weil height of $\eta$. If $\Lambda \neq 0$, then
\begin{equation}
\label{eq:211}
\nu_{\pi}(\Lambda)\leq 19(20\sqrt{t+1}D)^{2(t+1)}\cdot e_{\pi}^{t-1}\frac{p^{f_{\pi}}}{(f_{\pi}\log p)^2}\log (e^5tD)H_1\cdots H_t\log B.
\end{equation}
Here $f_{\pi}$ is the inertia degree of $\pi$, namely that positive integer such that the finite field ${\mathcal O}_{\mathbb K}/\pi$ has cardinality $p^{f_{\pi}}$. 
\end{theorem}
In our application, we take $t=2$, fix $i\in \{1,\ldots,s\}$ and put
$$
(\alpha_1,\alpha_2)=(\delta_i,\eta_i)\in \{(\beta,\gamma),(\beta,\beta)\},\quad (b_1,b_2)=(d_i,e_i)\in \{(n,1),(n,-1),(n,n)\},
$$ 
respectively according to which $\Lambda=\alpha_1^{b_1}\alpha_2^{-b_2}-1$ is nonzero. Thus, $B=n$ and $\alpha_1,\alpha_2\in \{\beta,\gamma\}$. 
Hence, we get
\begin{equation}
\label{eq:42}
\nu_{\pi}(\Lambda)\leq 19(20\sqrt{3}\cdot 4)^{6}\cdot e_{\pi}\frac{p^{f_{\pi}}}{(f_{\pi}\log 2)^2}\log (8e^5)H^2 \log n,
\end{equation}
where 
$$
H\geq \max\{h(\beta),h(\gamma),\log p\}.
$$
Since $\beta,\gamma$ are roots of the polynomials
$$
(b-1)^2X^2+(2a)X+1\qquad {\text{\rm and}}\qquad (b-1)^2X^2+(2c)X+1,
$$
are of degree at most $2$ and both these numbers as well as their conjugates are in absolute values at most 
$$
(b-1)^2/(b-1)+{\sqrt{((b-1)^2/(b-1))-1}}<2b,
$$
we conclude that 
$$
\max\{h(\beta),h(\gamma)\}<2\log(2b)\le 4\log b.
$$
Since $p\le b$, we can take $H=4\log b$. Furthermore, since $e_{\pi}\leq 4$ and $f_{\pi}\leq 4$, and $D\nu_p(n)\le 4( \log n)/(\log 2)<8\log n,$ inequalities \eqref{eq:30} and \eqref{eq:42} yield
\begin{equation}
\label{boundm}
m\leq 1.3\times 10^{17}b^4 (\log b)^2 \log n +16 \log n<2\times 10^{17} b^6 \log n.
\end{equation}
We record this as a lemma.
\begin{lemma}
\label{lem:11}
If system \eqref{eq:5} has a solution, then
$$m<2\times 10^{17} b^6\log n.$$
\end{lemma}

\section{Bounding $n$ in terms of $b$ }
  
The second equation of (\ref{eq:5}) gives
$$ 
\alpha^n + \alpha^{-n} = 2X_n=(2c/(b-1))b^{m\ell} - (2c/(b-1)). 
$$
Since
\begin{equation}
\label{eq:line16}
(2c/(b-1))b^{m\ell} - \alpha^n=\alpha^{-n} + (2c/(b-1)),
\end{equation}
the above leads to
\begin{equation}
\label{eq:23}
0< (2c/(b-1))b^{m\ell}\alpha^{-n} -1 < \frac{3}{\alpha^n}<\frac{1}{\alpha^{n-2}}.
\end{equation}
The left hand side of \eqref{eq:23} is non-zero by the equation \eqref{eq:line16}. We find a lower bound on it using a lower bound for a nonzero linear form in logarithms  of Matveev \cite{MV} which we now state.  
\begin{theorem}
In the notation of Theorem \ref{th1}, assume in addition that $\mathbb{K}$ is real and 
$$
H_i\geq \max\{Dh(\delta_i),|\log \delta_i|,0.16\}\quad {\text{for}}\quad i=1,\ldots,t. 
$$
If $\Lambda\neq 0$, then 
\begin{equation}
\label{eq:24}
\log |\Lambda|\geq -1.4 \cdot 30^{t+3}\cdot t^{4.5}\cdot D^2(1+\log D)(1+\log B)H_1\cdots H_t.
\end{equation}
\end{theorem}
We take $t=3$, $\delta_1=2c/(b-1)$, $\delta_2=b$, $\delta_3=\alpha$, $b_1=1$, $b_2=m\ell$, $b_3= -n$. Since $\ell\le 2n$ (see \eqref{eq:12}), we can take $B=2mn$. Now the algebraic numbers 
$\delta_1,\delta_2,\delta_3$ belong to ${\mathbb L}={\mathbb Q}[\alpha]$, a field of degree $D=2$. Since $h(\delta_1)\le \log(2 b)\le 2\log b$, we can take $H_1=4\log b$ and $H_2=2\log b$ . Furthermore, since
$\alpha$ is a quadratic unit, we can take $H_3=\log \alpha$.  Thus, we get, by \eqref{eq:23}, that
$$
(n-2)\log \alpha\le -\log \Lambda \leq 1.4\cdot 30^6\cdot 3^{4.5}\cdot 2^2\cdot (1+\log 2)(1+\log 2mn)(8(\log b)^2)\log \alpha,
$$
giving
$$
n \leq 10^{13}(1+\log (2mn)) (\log b)^2.
$$
Inserting  \eqref{boundm} into the above inequality we get
\begin{eqnarray}
\label{eq:100}
n & \leq & 10^{13}(1+\log (4\times 10^{17} b^6 n\log n))(\log b)^2\nonumber\\
& < & 10^{13}(1+\log(4\times 10^{17})+6\log b+2\log n)(\log b)^2\nonumber\\
& < & 10^{13}\cdot  43\cdot (6\log b) (2\log n) (\log b)^2\nonumber\\
& < & 6\times 10^{15} (\log b)^3 \log n.
\end{eqnarray}
In the above and in what follows we use the fact that if $x_1,\ldots,x_k$ are real numbers $>2$, then 
$$
x_1+\cdots+x_k\le x_1\cdots x_k.
$$
Lemma 1 in \cite{GL} says that if $T>3$ and
$$
\frac{n}{\log n}<T,\qquad {\text{\rm then}}\qquad n<(2T)\log T.
$$
Taking $T:=6\times 10^{15} (\log b)^3$ in the above implication and using \eqref{eq:100}, we get that
\begin{eqnarray*}
n & < & 12\times 10^{15} (\log b)^3 \log(6\times 10^{15} (\log b)^3)\\
& < & 12\times 10^{15} (\log b)^2(\log(6\times 10^{15})+3\log b)\\
& < & 12\times 10^{15} (\log b)^3 \times 37\times (3\log b)\\
& < & 2\times 10^{18} (\log b)^4.
\end{eqnarray*}
Inserting this back into the inequality of Lemma \ref{lem:11}, we get
\begin{eqnarray*}
m & \le &  2\times 10^{17} b^6 \log(2\times 10^{18} (\log b)^4)\\
& = & 2\times 10^{17}b^6(\log(2\times 10^{18})+4\log b)\\
& < & 2\times 10^{17} b^6 \times 43\times 4\log b\\
& < & 10^{20} b^7.
\end{eqnarray*}
Since $\ell\le 2n$ (see \eqref{eq:12}), we conclude the following result.
\begin{lemma}
\label{lem:12}
Under the hypothesis of Problem \ref{prob}, we have that
$$
n<10^{18} b^4,\qquad  \ell<2\times 10^{18} b^4, \qquad m<10^{20} b^7.
$$
\end{lemma}  
Finally, from \eqref{eq:5} we have that $d<X_1^2<b^{2m}$, therefore
\begin{equation*}
\label{aq:10bis}
d<b^{2\times 10^{20} b^7}<\exp(10^{20} b^{10}),
 \end{equation*}
which together with Lemma \ref{lem:2} implies the desired conclusion of Theorem \ref{thm:main}.

\section*{Acknowledgements}

We thank the referee for spotting a wrong calculation in a previous version of this manuscript.

\end{document}